\documentclass[reqno, a4paper,11pt]{amsart}
\usepackage[english]{babel}
\usepackage[utf8]{inputenc}

\usepackage{color,bm}
\usepackage{enumerate}
\usepackage{amsthm,amssymb,mathrsfs}
\usepackage{amsmath}
\usepackage{graphicx,epsfig}
\usepackage{hyperref}
\usepackage{amsfonts, mathtools, subfig, makebox, dsfont}
\usepackage{comment}

\hypersetup{hidelinks}

\allowdisplaybreaks[1]

\def\C{{\mathbb{C}}}

\def\R{{\mathbb{R}}}

\def\N{\mathbb N}

\def\K{\mathbb K}

\newtheorem{theo}{Theorem}[section]
\newtheorem{lemma}[theo]{Lemma}
\newtheorem{prop}[theo]{Proposition}
\newtheorem{cor}[theo]{Corollary}

\newtheorem{hyps}[theo]{Hypotheses}
\theoremstyle{definition}

\newtheorem{rem}[theo]{Remark}

\def\f{\bm{f}}
\def\uu{\bm{u}}

\begin{document}
\numberwithin{equation}{section}
\title[Vectorial $\textup{H}^\infty$-calculus for generators]{Bounded $\textup{H}^\infty$-calculus for vectorial-valued operators with Gaussian kernel estimates}
\author[D. Addona]{Davide Addona}
\author[V. Leone]{Vincenzo Leone}
\author[L. Lorenzi]{Luca Lorenzi} 
\author[A. Rhandi]{Abdelaziz Rhandi}


\address{D.Addona, L.Lorenzi: Plesso di Matematica, Dipartimento di Scienze Matematiche, Fisiche e Informatiche, Università di Parma, Parco Area delle Scienze 53/A, 43124 Parma, Italy}
\email{davide.addona@unipr.it, luca.lorenzi@unipr.it}

\address{V.Leone, A.Rhandi: Dipartimento di Matematica, Università degli Studi di Salerno, Via Giovanni Paolo II, 132, 84084 Fisciano (SA), Italy}
\email{vleone@unisa.it, arhandi@unisa.it}

\vskip 0.5cm
\keywords{Bounded $\textup{H}^\infty$-calculus, Vector-valued elliptic operators, Vector-valued analytic semigroups, Complex Gaussian kernel estimates}
 
\subjclass[2020]{Primary: 47A60;  Secondary: 35J47, 35K08, 47D06}
	
\begin{abstract}
We prove that the vector-valued generator of a bounded holomorphic semigroup represented by a kernel satisfying  Gaussian estimates with bounded $\textup{H}^\infty$-calculus in $L^2(\R^d;\C^m)$  admits bounded $\textup{H}^\infty$-calculus for every $p\in (1,\infty)$.
We apply this result to the elliptic operator $-\operatorname{div}(Q\nabla)+V$, where the potential term V is a matrix-valued function whose entries belong to $L^1_{\rm loc}(\R^d)$ and, for almost every $x\in \R^d$, $V(x)$ is a symmetric and nonnegative definite matrix.
\end{abstract}
	
\maketitle

\section{Introduction}
The $\textup{H}^\infty$-calculus for sectorial operators, a concept introduced by McIntosh \cite{M86}, is a powerful functional calculus that extends the standard holomorphic functional calculus to a broader class of unbounded operators, cf. \cite{HNVW17} and the references therein. It plays a crucial role in evolution equations, control theory, and harmonic analysis. More precisely, it provides a powerful framework to establish maximal regularity for sectorial operators, which ensures well-posedness and stability of solutions. It is used in robust control problems where sectorial operators describe stability properties of dynamical systems and in harmonic analysis, which is connected to Riesz transforms and spectral multipliers.  It also allows to establish off-diagonal decay estimates and $L^p$-boundedness of spectral multipliers (see \cite{CDMY96}, \cite{Kriegler-thesis} and the references therein). 

The generator $-L$ of a bounded holomorphic semigroup is a sectorial operator and, in many situations, it becomes crucial to determine whether, for a given holomorphic function $\xi$, the operator $\xi(L)$ is defined and bounded. 
In Hilbert spaces, this is indeed the case, as a consequence of maximal accretivity and von Neumann’s inequality; see \cite[Chapter 7]{Haase06} for further details. Beyond the Hilbert space setting, well-known examples of generators admitting a bounded  $\textup{H}^\infty$-calculus include generators of $C_0$-groups on $L^p$-spaces  as well as negative generators of positive, strongly continuous semigroups on $L^p$-spaces. These results follow from the transference principle of Coifman and Weiss \cite{CoWe77}. In particular, this implies the boundedness of the imaginary powers of such operators, which has several applications, for instance in complex interpolation theory and in establishing bounds for the Riesz transform. 

Another strategy was developed by Duong and Robinson \cite{DuRo96}, who considered operators with an $L^2$-bounded $\textup{H}^\infty$-calculus that generate a bounded holomorphic semigroup represented by a kernel satisfying a Poisson-type bound for complex time, a condition that includes the typical Gaussian bounds associated with elliptic operators. They showed that such operators admit a $L^p$-bounded $\textup{H}^\infty$-calculus as well. In this framework, we also cite the work \cite{heb}, which proved a multiplier theorem for Schr\"odinger operators.

The situation is quite different for vector-valued operators. In fact, while in Hilbert spaces the abstract theory allows us to conclude that maximal accretive operators in $L^2(\R^d;\C^m)$ have a bounded $\textup{H}^\infty$-calculus, the validity of the transference principle for the generator of a positive semigroup in $\R^m$-valued $L^p$-spaces, or more in general $Y$-valued $L^p$-spaces (when $Y$ is a Banach space), is only known for semigroups of the form $T {\rm Id}_Y$, where $T$ is a positive contraction semigroup on scalar $L^p$-spaces, see e.g. \cite{HP98}.

In this paper, we extend the aforementioned strategy of \cite{DuRo96} to a broader class of vector-valued operators and establish the boundedness of the $\textup{H}^\infty$-calculus in $\C^m$-valued $L^p$-spaces. Specifically, we show that the combination of a Gaussian-type kernel estimates and a Calder\'on-Zygmund decomposition for scalar functions yields that the operator $\xi(L)$ is well-defined and weakly bounded in $L^1(\R^d;\C^m)$ when $\xi$ is a bounded and holomorphic function with polynomial decay at both zero and infinity. This decay condition can be removed through a density argument. Finally, the result follows from the  Marcinkiewicz interpolation theorem. We conclude with an application of this theory to a vector-valued Schr\"odinger type operator.

\section{Notation, basic definitions and preliminary results}

\subsection{Notation} 
Given a set $E\subset\R^d$, we denote by $\chi_E$ its characteristic function.
We find it convenient to display vector-valued functions in bold style. If $\f$ is any such function, with values in $\R^m$ (resp. $\C^m$), then we denote by $f_1,\ldots,f_m$ its components.

We denote by $C_c^{\infty}(\R^d;\mathbb{K}^m)$ ($d,m\in\N)$ the space of the compactly supported and infinitely many times differentiable functions $\f:\R^d\to\mathbb{K}^m$ (here, $\K$ is either $\R$ or $\C$) and, for $p\in[1,\infty)\cup\{\infty\}$,  $L^p(\R^d;\mathbb{K}^m) $ denotes the space of (the equivalent classes of) measurable $\mathbb K^m$-valued functions, such that
$\|\f\|_p= \left(\int_{\R^d}\|\f(x)\|^pdx\right)^{\frac1p}$
 is finite, if $1 \leq p <\infty$, and such that
$\|\f\|_{\infty}= {\rm ess\sup}_{x\in\R^d}\|\f(x)\|<\infty$, if $p=\infty.$ In particular, we denote with 
$L^p_{\rm loc}(\R^d;\mathbb{K}^m)$ the set of (equivalent classes of) measurable functions that belong $L^p(\Omega,\mathbb{K}^m)$ for every bounded measurable subset $\Omega$ of $\R^d$.

The Sobolev space $H^1(\R^d;\mathbb{K}^m)$ is the set of functions $\f\in L^2(\R^d;\mathbb{K}^m)$ such that the first-order distributional derivatives belong to $L^2(\R^d;\mathbb{K}^m)$ for any multi-index $\beta$ with length $1$. It is endowed with its natural norm.

For $\theta \in (0, \pi)$, $\Sigma_{\theta}$ denotes the open sector of angle $\theta$, centered at the origin, i.e., $\Sigma_\theta \coloneqq \{ \lambda \in \C \setminus \{0\}\colon |\arg(\lambda) | < \theta \}$, and $\Sigma_0\coloneqq(0,\infty)$. 
For a bounded function  $\xi\colon \Sigma_\theta \to \mathbb{C}$, we denote by $\| \xi \|_{\infty, \theta}$ its sup-norm. Furthermore, we denote by $\textup{H}(\Sigma_\theta)$ the space of all holomorphic functions on $\Sigma_\theta$.
We also introduce the spaces
$\textup{H}^\infty(\Sigma_\theta)\coloneqq\textup{H}(\Sigma_\theta)\cap C_b(\Sigma_{\theta})$ and
\begin{equation*}
\Psi(\Sigma_\theta)\coloneqq\left\{\xi\in \textup{H}^\infty(\Sigma_\theta)\colon\exists\ c,s>0\colon |\xi(z)|\leq c\frac{|z|^s}{1+|z|^{2s}} \ \forall z\in \Sigma_\theta\right\}. 
\end{equation*}

The open ball in $\R^d$ centered at $x$ with radius $r$ is denoted by $B(x,r)$ and $|B(x,r)|$ stands for its Lebesgue measure. When $x=0$, we simply write $B(r)$ instead of $B(0,r)$. 

The space of all bounded linear operators from a Banach space $X$ into another Banach space $Y$ is denoted by $\mathscr{L}(X,Y)$. If $Y=X$ then we simply write $\mathscr L(X)$. Finally, by $\bm{e}_k$ we denote the  $k$-th element of the canonical basis of $\K^\ell$, $\ell=d$ or $\ell=m$, and, by $d(E,F)$ we denote the distance between the sets $E, F\subset\R^d$. 

\subsection{Basic definitions} 
A linear operator $L$ on a Banach space $X$ is called \emph{sectorial of angle} $\theta \in [0, \pi)$ if its spectrum $\sigma(L)$ is contained in the closure of $\Sigma_\theta$ and there exists a constant $C_\mu > 0$ such that, for all $\mu \in (\theta, \pi)$, it holds that
\begin{equation*}
\| R(\lambda, L) \|_{{\mathscr L}(X)} \leq \frac{C_\mu}{|\lambda|}, \qquad\;\,\lambda\in\mathbb{C}, \ |\arg(\lambda)|>\mu. 
\end{equation*}		

For $\xi\in\Psi(\Sigma_\mu)$, we define the operator $\xi(L)\in{\mathscr L}(X)$ by setting
\begin{equation}
\label{xi(L)}
\xi(L) \coloneqq \frac1{2\pi i} \int_{\gamma_\varphi} \xi(\lambda)R(\lambda,L)d\lambda,
\end{equation}
where  $\gamma_\varphi$ is the boundary of $\Sigma_\varphi$ oriented counterclockwise and $\varphi\in (\theta,\mu)$. 
        
The sectoriality of $L$ and the properties of the functions in the class $\Psi(\Sigma_\mu)$ ensure that the integral in the right-hand side of \eqref{xi(L)} is absolutely convergent in the operator norm topology and it is independent of the choice of $\varphi$.

Furthermore, if $L$ is sectorial and injective, it is possible to define $\xi(L)$ for every function $\xi\in\textup{H}^\infty(\Sigma_\mu)$ and the new definition agrees with the above one when $\xi\in \Psi(\Sigma_{\mu})$. Note that, if $\xi\in \textup{H}^{\infty}(\Sigma_{\mu})$, then, in general, $\xi(L)$ is not bounded. We say that $L$ admits a bounded $\textup{H}^\infty$-calculus of angle $ \mu \in  (\theta, \pi) $ if there exists a constant $ C_\mu>0$ such that for all $ \xi \in \textup{H}^\infty(\Sigma_\mu)$ $\xi(L)$ belongs to ${\mathscr L}(X)$, with the norm bound
\begin{equation*}
\| \xi(L) \|_{{\mathscr L}(X)}\le C_\mu\|\xi\|_{\infty, \mu}. 
\end{equation*}
For a detailed discussion about this construction, we refer to \cite[Chapter 5]{Haase06}. 
In particular, if $X$ is a Hilbert space and $L$ is sectorial, self-adjoint and injective, then it admits a bounded $\textup{H}^{\infty}$-calculus of any angle 
$\mu\in(0,\pi)$ and for every $\xi\in \textup{H}^\infty(\Sigma_\mu)$, we can define the operator $\xi(H)$ by the Dunford calculus (see \cite[Corollary 7.1.6]{Haase06} and \cite[Subsection 10.2.a]{HNVW17}).

There is a one-to-one correspondence between sectorial operators and bounded holomorphic semigroups. In particular, if $L$ is sectorial of angle $\theta\in\left[0,\frac\pi2\right)$, then the family $(e^{-zL})_{z\in \Sigma_{\frac\pi2-\theta}}$ is a bounded holomorphic semigroup. On the other hand, if $(T(z))_{z\in \Sigma_\theta}$ is a bounded holomorphic semigroup for some $\theta\in\left(0,\frac\pi2\right]$, then the unique operator $A$ such that 
$(\lambda+A)^{-1}=\int_0^\infty e^{-\lambda t}T(t)dt$ for every $\lambda\in\C$ with positive real part, is sectorial of angle $\frac\pi2-\theta$ and $T(z) = e^{-zA}$ for every $z\in \Sigma_\theta$ (see \cite[Proposition 3.4.4]{Haase06}).

The Hardy-Littlewood maximal 
function $Mf$ is defined 
on scalar functions $f\in L_{\rm loc}^1(\R^d)$ by
\begin{equation*}
(Mf)(x)\coloneqq\sup_{r>0}\frac{1}{|B(x,r)|}\int_{B(x,r)}|f(y)|dy,\qquad\;\, x\in\R^d.
\end{equation*}

For every $t>0$ and $x\in\R^d$, we define the Gaussian kernel $w$ as 
\begin{equation}\label{eq:def:gaussian:fun}
w(t,x)\coloneqq c_1t^{-\frac d2}\exp\bigg (-c_2\frac{|x|^2}{t}\bigg ),
\end{equation}
where $c_1,c_2$ are positive constants. 

\subsection{Preliminary results} To prove our main results, we need some technical estimates which will be exploited during the proof of Theorem \ref{thm:Hinfinity_calculus_Lp:gen}. For a detailed proof of these results (also for more general kernels $w$), we refer to \cite[Propositions 2.1, 2.4 \& 2.5]{DuRo96}. 
\begin{lemma}
\label{lemma:kernel_est} 
The function $w$ introduced in \eqref{eq:def:gaussian:fun} satisfies the following properties.
\begin{enumerate}[\rm(i)]
\item For every $\gamma>0$ there exists a positive constant $c$ such that
\begin{align*}
\int_{\R^d\setminus B(r)}w(t,y)dy\leq \frac{c}{(1+r^2t^{-1})^{\gamma}}, \qquad\;\,t>0,\;\, r\geq 0.    
\end{align*}
\item For every $p\in(1,\infty)$ and $t>0$, there exists a positive constant $c$ such that
\begin{align*}
\int_{\R^d}w(t,x-y)\|\bm{u}(y)\|dy\leq cM(\|\bm{u}\|)(x), \qquad \textrm{for a.e. }x\in \R^d, \ {\bm{u}}\in L^p(\R^d;\C^m).
\end{align*}
\item For every $v>0$, there exist $c,\tau\geq 1$ such that
\begin{align*}
\sup_{z\in B(y,r)}w(t,x-z)
\leq c\inf_{z\in B(y,r)}w(\tau t,x-z),
\end{align*}
uniformly with respect to $x,y\in \R^d$ and to $r,t>0$ with $r\le\sqrt{vt}$.
\end{enumerate}
\end{lemma}


\begin{lemma}
\label{lemma:approssimazione}
Let $S\colon\mathcal D\to \mathcal M(\R^d;\C^m)$ be a linear operator, where $\mathcal D$ is a dense subspace of $L^1(\R^d;\C^m)$ and $\mathcal M(\R^d;\C^m)$ is the space of measurable functions from $\R^d$ into $\C^m$. Assume that there exists a positive constant $c$, independent of $\alpha$, such that
\begin{align}
\label{weak_11_dense}
|\{x\in \R^d:\|(S\bm{u})(x)\|>\alpha\}|
\leq \frac{c}{\alpha}\|\bm{u}\|_1, \qquad\;\, \alpha>0,
\end{align}
for every $\bm{u}\in \mathcal D$. Then, $S$ uniquely extends to an operator defined on $L^1(\R^d;\C^m)$, which we still denote by $S$, and \eqref{weak_11_dense} is satisfied by every $\bm{u}\in L^1(\R^d;\C^m)$.
\end{lemma}

\begin{proof}
Fix ${\bm{u}}\in L^1(\R^d;\C^m)$ and let $(\bm{u}_n)\subseteq\mathcal D$ be a sequence converging to ${\bm{u}}$ in $L^1(\R^d;\C^m)$. From \eqref{weak_11_dense} it follows that, for every $k,n\in\N$,
\begin{align}
\label{weak_11_difference}
|\{x\in \R^d\colon\|(S(\bm{u}_n-\bm{u}_k))(x)\|>\alpha\}|
\leq \frac{c}{\alpha}\|\bm{u}_n-\bm{u}_k\|_1, \qquad\;\, \alpha>0.
\end{align}

We claim that there exists a subsequence $(\bm{u}_{n_k})\subseteq (\bm{u}_n)$ such that $(S\bm{u}_{n_k})$ converges pointwise almost everywhere in $\R^d$ (say in the set $A\subset\R^d$) and the limit does not depend on the sequence $(\bm{u}_n)$. Once the claim is proved, we
can easily conclude the proof denoting by $S\bm{u}$ the limit of $(S\bm{u}_{n_k})$. Indeed,
\begin{align*}
|\{x\in \R^d\colon \|(S\bm{u})(x)\|>\alpha\}|=&|\{x\in A\colon \|(S\bm{u})(x)\|>\alpha\}|\\
\le &\bigg |\bigcup_{k\in\N}\bigcap_{h\ge k}\{x\in A\colon \|(S\bm{u}_{n_h})(x)\|>\alpha\}\bigg | \\
= & \lim_{k\to\infty}\bigg |
\bigcap_{h\ge k}\{x\in A\colon \|(S\bm{u}_{n_h})(x)\|>\alpha\}\bigg | 
\le \frac{c}{\alpha}\lim_{k\to\infty}\|\bm{u}_{n_k}\|_1\\
=&\frac{c}{\alpha}\|\bm{u}\|_1.
\end{align*}
Since the proof of the claim is rather long, we split it into three steps.

{\em Step 1}. We fix $\alpha>0$ and denote by $B^+(1)$ (resp. $B^-(1)$) the set of all
$x\in B(1)$ such that $\|(S\uu_n)(x)-(S\uu_k)(x)\|>\frac{\alpha}{2|B(1)|}$
(resp.
$\|(S\uu_n)(x)-(S\uu_k)(x)\|\le\frac{\alpha}{2|B(1)|}$). Finally, we introduce the function $\phi:[0,\infty)\to [0,\infty)$, defined by $\phi(s)=s(s+1)^{-1}$ for every $s\ge 0$.

We infer that, for every $x\in B(1)$ and every $k,n\in\N$
\begin{align*}
\phi(\|(S\bm{u}_n)(x)-(S\bm{u}_k)(x)\|)
= & \chi_{B^+(1)}(x)\phi(\|(S\bm{u}_n)(x)-(S\bm{u}_k)(x)\|)\\
&+\chi_{B^-(1)}(x)\phi(\|(S\bm{u}_n)(x)-(S\bm{u}_k)(x)\|) \\
\leq & \chi_{B^+(1)}(x)+\frac{\alpha}{2|B(1)|}.
\end{align*}
Integrating on $B(1)$, from \eqref{weak_11_difference} it follows that
\begin{align*}
\|\phi(\|S\bm{u}_n-S\bm{u}_k\|)\|_{L^{1}(B(1);\R)}
\leq & \frac{2c|B(1)|}{\alpha}\|\bm{u}_n-\bm{u}_k\|_1+\frac{\alpha}{2}.
\end{align*}
The convergence of $(\bm{u}_n)$ to ${\bm{u}}$ in $L^1(\R^d;\C^m)$ implies that there exists $\overline n\in\N$ such that
\begin{align}
\label{stima_palla_alpha}
\int_{B(1)}\phi(\|(S\bm{u}_n)(x)-(S\bm{u}_k)(x)\|)dx\leq \alpha
\end{align}
for every $k,n\in\N$, with $k,n\geq \overline n$.
Choosing $\alpha=2^{-\ell}$, with $\ell\in\N$, in \eqref{stima_palla_alpha}, we infer that there exists an increasing sequence $(n_\ell^1)$ of integers such that 
\begin{align*}
\int_{B(1)}\phi(\|(S\bm{u}_{n_{\ell+1}^1})(x)-(S\bm{u}_{n_\ell^1})(x)\|)dx
\leq \frac1{2^\ell}, \qquad \ell\in\N.
\end{align*}
This implies that 
\begin{align*}
\sum_{\ell=1}^{\infty}\int_{B(1)}\phi(\|(S\bm{u}_{n_{\ell+1}^1})(x)-(S\bm{u}_{n_\ell^1})(x)\|)dx<\infty
\end{align*}
and we can determine a measurable set $A_1\subseteq B(1)$, with $|A_1|=|B(1)|$, such that  
\begin{align}
\label{conv_serie_punt_1}
\sum_{\ell=1}^{\infty}\phi(\|(S\bm{u}_{n_{\ell+1}^1})(x)-(S\bm{u}_{n_\ell^1})(x)\|)<\infty,\qquad\;\,x\in A_1.
\end{align}

Fix $x\in A_1$. From \eqref{conv_serie_punt_1} we deduce that there exists $\overline \ell\in\N$ such that $\phi(\|(S\bm{u}_{n_{\ell+1}^1})(x)-(S\bm{u}_{n_\ell^1})(x)\|)\leq \frac12$ for every $\ell\geq \overline \ell$. Hence,
\begin{align*}
\|(S\bm{u}_{n_{\ell+1}^1})(x)-(S\bm{u}_{n_\ell^1})(x)\|=&\phi^{-1}(\phi(\|(S\bm{u}_{n_{\ell+1}^1})(x)-(S\bm{u}_{n_\ell^1})(x)\|))\\
\leq & 
\phi(\|(S\bm{u}_{n_{\ell+1}^1})(x)-(S\bm{u}_{n_\ell^1})(x)\|)
\end{align*}
for every $\ell\ge\overline\ell$, and the series
$\sum_{\ell=1}^\infty\|(S\bm{u}_{n_{\ell+1}^1})(x)-(S\bm{u}_{n_\ell^1})(x)\|$ converges. Using the generalized triangle inequality we conclude that  $(S\bm{u}_{n_{\ell}^1}(x))$ is a Cauchy sequence in $\C^m$ for every $x\in A_1$. We denote by $\bm{v}_1(x)$ its limit. 

{\em Step 2}. Repeating the same computations as in Step $1$, with the unit ball being replaced by the ball centered at the origin with radius $j\in\N$, we can determine a net of subsequences $(n_\ell^1)\supset(n_\ell^2)\supset\ldots$ such that for every $j\in\N$ the sequence $(S\bm{u}_{n_\ell^j})$ converges pointwise, in a set $A_j\subseteq B(j)$, with $|A_j|=|B(j)|$, to a function denoted by $\bm{v}_j$. Clearly, $\bm{v}_j=\bm{v}_k$ on $A_k$ for every $j,k\in\N$, with $k\leq j$, and this allows us to define a function $\bm{v}$ by setting $\bm{v}(x)=\bm{v}_j(x)$ for every $x\in A_j$ and every $j\in\N$. Since the set $N=\cup_{j\in\N}(B_j\setminus A_j)$ 
has null Lebesgue measure, it follows that $\bm{v}$ is defined almost everywhere in $\R^d$.

We finally set $\bm{w}_k=\bm{u}_{n_k^k}$ for every $k\in\N$.  For every $x\in \R^d\setminus N$ there exists $j\in\N$ such that $x\in A_j$ and therefore $(S\bm{w}_k(x))$ converges to $\bm{v}(x)$ since $(\bm{w}_k)\subset (\bm{u}_{n_\ell^j})$ definitely.

{\em Step 3.} We conclude the proof by showing that, for every $\bm{u}\in L^1(\R^d;\C^m)$, the limit ${\bm{v}}$ is independent of  the choice of the approximating sequence $(\bm{u}_n)$. To this aim, let  $(\widetilde {\bm{u}}_n)$ be another sequence in $\mathcal D$ which converges to $\bm{u}$ in $L^1(\R^d;\C^m)$ and let $(\widetilde{\bm{w}}_k)$ and $\widetilde{\bm{v}}$ be the functions obtained by the arguments in the previous two steps starting from the sequence $(\widetilde {\bm{u}}_n)$. From \eqref{weak_11_difference}, with $\bm{u}_n$ and $\uu_k$ replaced, respectively, by $\bm{w}_j$ and  $\widetilde{\bm{w}}_j$, we infer that
\begin{align*}
\left|\left\{x\in\R^d\colon\|(S\bm{w}_j)(x)-(S\widetilde{\bm{w}}_j)(x)\|>\alpha\right\}\right|\leq \frac{c}{\alpha}\|\bm{w}_j-\widetilde{\bm{w}}_j\|_1, \qquad\;\, \alpha>0, \;\, j\in\N.    
\end{align*}
As $j$ tends to infinity, we deduce that
$|\{x\in\R^d:\|\bm{v}(x)-\widetilde{\bm{v}}(x)\|>\alpha\}|=0$ for every $\alpha>0$,   
which implies that $\bm{v}=\widetilde{\bm{v}}$ almost everywhere in $\R^d$. The proof is complete.
\end{proof}

\section{The boundedness of the \texorpdfstring{$\textup{H}^\infty$}{H-infinity}-calculus in \texorpdfstring{$L^p$}{Lp}}

Consider now a vector-valued operator $L$ such that $-L$ generates a strongly continuous semigroup $(e^{-tL})_{t\ge0}$ in $L^2(\R^d;\C^m)$. This means that 
\begin{align*}
e^{-tL}\bm{u}=\sum_{i,j=1}^m(T_{ij}(t)u_j){\bm e}_i,  \qquad\;\,t>0,\;\,\uu\in L^2(\R^d;\C^m),
\end{align*}
where, for every $i,j\in\{1,\dots,m\}$ and $t>0$, $T_{ij}(t)f = \left\langle e^{-tL}(f{\bm e}_j),{\bm e}_i\right\rangle$ for every $f\in L^2(\R^d;\C)$. 


We introduce the following assumptions.

\begin{hyps}
\label{hyp:operatore}
Let $L$ and $(e^{-tL})_{t\ge0}$ be as above and suppose that:
\begin{enumerate}[\rm(i)]
\item 
$(e^{-tL})_{t\ge0}$ extends to a bounded holomorphic semigroup in $L^2(\R^d;\C^m)$ in the sector $\Sigma_\theta$ for some $\theta\in\left(0,\frac\pi2\right]$. Then, $L$ is sectorial of angle $\frac\pi2-\theta$;
\item 
$L$ has a bounded $\textup{H}^\infty(\Sigma_\mu)$-calculus on $L^2(\R^d;\C^m)$ for some $\mu>\frac\pi2-\theta$;
\item 
for every $t>0$ and every $i,j\in\{1,\ldots,m\}$ the operator $T_{ij}(t)$ admits a kernel $p_{ij}(t,\cdot ,\cdot)$, i.e.,
\begin{equation*}
(T_{ij}(t)f)(x)=\int_{\R^d}p_{ij}(t,x,y)f(y)dy,\qquad\;\, \textit{for a.e. } x\in \R^d,\;\,f\in L^1(\R^d;\C),  
\end{equation*}
and
\begin{equation}
\label{stima_ker_smgr_gen}
|p_{ij}(t,x,y)|\le w(t,x-y), \qquad\;\, i,j\in\{1,\ldots,m\},
\end{equation}
for almost every $(x,y)\in \R^d\times \R^d$.
\end{enumerate}
\end{hyps}

The matrix $(p_{ij}(t,\cdot,\cdot))_{1\le i,j\le m}$ is called the matrix kernel associated to the semigroup $(e^{-tL})_{t\ge0}$. 

\begin{rem}
\label{rmk:adjoint}
We stress that the adjoint semigroup $(e^{-tL^*})_{t\geq0}$ is strongly continuous on $L^2(\R^d;\C^m)$ and admits a matrix of kernels whose entries $p^*_{ij}(t,\cdot,\cdot)$ satisfy, for every $t>0$, the condition $p^*_{ij}(t,x,y) = p_{ji}(t,y,x)$ for a.e. $(x,y)\in\R^d\times\R^d$ and $i,j\in\{1,\dots,m\}$. In particular, since also $L^*$ is sectorial of angle $\frac\pi2-\theta$ and $L$ admits a bounded $\textup{H}^\infty$-calculus on $L^2(\R^d;\C^m)$ if and only if $L^*$ does on the same sector, all the assumptions of this section are invariant under duality.
\end{rem}

\begin{prop}
\label{prop-3.3}
Under Hypotheses $\ref{hyp:operatore}$, for every $p\in[1,\infty)$, $(e^{-tL})_{t\geq0}$ uniquely extends to a bounded, strongly continuous semigroup $(e^{-tL_p})_{t\geq0}$ in $L^p(\R^d;\C^m)$ and we denote by $-L_p$ its infinitesimal generator.   
\end{prop}

\begin{proof}
Throughout the proof, by $c$ we denote a positive constant, which depends at most on $d$, $m$, $p$, $c_1$ and $c_2$ and may vary from line to line.

We fix $p\in [1,\infty)$ and observe that, from Hypothesis \ref{hyp:operatore}(iii), it follows that 
\begin{align*}
\|e^{-tL}\bm{f}\|_p^p
=&\int_{\R^d}\|(e^{-tL}\bm{f})(x)\|^pdx
\leq  c\sum_{k=1}^m\int_{\R^d}\left(\int_{\R^d}w(t,x-y)|f_k(y)|dy\right)^pdx \\
\leq & c\|w(t,\cdot)\|_1^p\sum_{k=1}^m\|f_k\|_p^p
\end{align*}
for every $\bm{f}\in C^\infty_c(\R^d;\C^m)$. This implies that, for every $t>0$, the operator $e^{-tL}$ uniquely extends to a linear bounded operator, which we denote by $T_p(t)$, on $L^p(\R^d;\C^m)$, and it satisfies
\begin{align*}
\|T_p(t)\bm{f}\|_p\leq c\|\bm{f}\|_p, \qquad\;\, \bm{f}\in L^p(\R^d;\C^m).
\end{align*}
Moreover,  
\begin{align*}
\|e^{-tL}\bm{f}\|_{\infty}
\leq & m\|w(t,\cdot)\|_{1}\sum_{k=1}^m\|f_k\|_{\infty}\le c\|\f\|_{\infty},\qquad\;\,\f\in C^{\infty}_c(\R^d;\C^m).
\end{align*}

Let us prove that $(T_p(t))_{t\geq0}$ is a strongly continuous semigroup on $L^p(\R^d;\C^m)$ for every $p\in[1,\infty)$. 
If $p>2$, then for every $\bm{f}\in C^\infty_c(\R^d;\C^m)$, we get
\begin{align*}
\lim_{t\to 0^+}\|T_p(t)\bm{f}-\bm{f}\|_p
&\leq \lim_{t\to0^+}\|e^{-tL}\bm{f}-\bm{f}\|_2^\frac2p\|e^{-tL}\bm{f}-\bm{f}\|_\infty^{1-\frac2p}\\
&\leq c\|\bm{f}\|_\infty^{1-\frac2p}\lim_{t\to 0^+}\|e^{-tL}\bm{f}-\bm{f}\|_2^\frac{2}{p}=0.
\end{align*}

Since the family of operators $(T_p(t))_{t\ge 0}$ is uniformly bounded and $C^{\infty}_c(\R^d;\C^m)$ is dense in $L^p(\R^d;\C^m)$ we can infer that $(T_p(t)\f)_{t\geq0}$ converges to $\f$ in $L^p(\R^d;\C^m)$ as $t$ tends to $0^+$ for every $\f\in L^p(\R^d;\C^m)$.


Similar arguments hold for $p\in(1,2)$, interpolating $L^p(\R^d;\C^m)$ between $L^1(\R^d;\C^m)$ and $L^2(\R^d;\C^m)$, and show that also for $p$ varying in this range $(T_p(t)\bm{f})_{t\geq0}$ converges to $\bm{f}$ in $L^p(\R^d;\C^m)$ as $t$ tends to $0^+$ for every $\bm{f}\in L^p(\R^d;\C^m)$. 

It remains to consider the case $p=1$.  Here, we take advantage of Lemma \ref{lemma:kernel_est}(i). Indeed, for every $\bm{f}\in C^\infty_c(\R^d;\C^m)$ with support contained in $B(r)$ for some $r>0$ we get
\begin{align*}
\int_{B(r+1)}\|(e^{-tL}\bm{f})(x)-\bm{f}(x)\|dx\leq |B(r+1)|^\frac{1}{2}\|e^{-tL}\bm{f}-\bm{f}\|_2.    
\end{align*}
Moreover, for every $h\in \{1,\ldots,m\}$ it follows that
\begin{align*}
\int_{\R^d\setminus B(r+1)}|(e^{-tL}\bm{f}(x))_h-f_h(x)|dx 
= & \int_{\R^d\setminus B(r+1)}|(e^{-tL}\bm{f}(x))_h|dx \\
= & \int_{\R^d\setminus B(r+1)}\bigg |\sum_{k=1}^m\int_{B(r)}p_{hk}(t,x,y)f_k(y)dy\bigg |dx \\
\leq & \sum_{k=1}^m\int_{\R^d\setminus B(r+1)}\bigg (\int_{B(r)}w(t,x-y)|f_k(y)|dy\bigg )dx.
\end{align*}
By applying Lemma \ref{lemma:kernel_est}(i) with $\gamma=1$, we infer that
\begin{align*}
&\sum_{k=1}^m\int_{\R^d\setminus B(r+1)}\bigg (\int_{B(r)}w(t,x-y)|f_k(y)|dy\bigg )dx\\
\leq & \sum_{k=1}^m\int_{B(r)}|f_k(y)|dy\int_{\R^d\setminus B(y,1)}w(t,x-y)dx \\
\leq & \frac{ct}{t+1}\sum_{k=1}^m\int_{\R^d}|f_k(y)|dy \\
\leq & \frac{c t}{t+1}\|\bm{f}\|_1.
\end{align*}
Letting $t$ tend to $0^+$, from the strong continuity of $(e^{-tL})_{t\geq0}$ in $L^2(\R^d;\C^m)$ it follows that
\begin{align*}
&\lim_{t\to0^+}\|e^{-tL}\bm{f}-\bm{f}\|_1\\
= & \lim_{t\to0^+}\bigg (\int_{B(r+1)}\|e^{-tL}\bm{f}(x)-\bm{f}(x)\|dx+\int_{\R^d\setminus B(r+1)}\|e^{-tL}\bm{f}(x)-\bm{f}(x)\|dx\bigg ) \\
\leq & \lim_{t\to0^+}\bigg(|B(r+1)|^{\frac12}\|e^{-tL}\bm{f}-\bm{f}\|_2+\frac{c t}{t+1}\|\bm{f}\|_1\bigg)=0.
\end{align*}
Again, the density of $C^\infty_c(\R^d;\C^m)$ in $L^1(\R^d;\C^m)$ and the fact that $(e^{-tL})_{t\geq0}$ is bounded in $L^1(\R^d;\C^m)$ yield the thesis. 

Finally, the semigroup property of $(T_p(t))_{t\geq0}$ clearly holds true for $\bm{f}\in C^\infty_c(\R^d;\C^m)$ and, by density, extends to every $\bm{f}\in L^p(\R^d;\C^m)$ with $p\in[1,\infty)$. 
\end{proof}

Thanks to Hypothesis \ref{hyp:operatore}(iii), we deduce the following complex Gaussian kernel estimates.

\begin{prop}
\label{prop:cplx_kernel_est_gen}
For every  $\beta\in\left(0,\theta\right)$, there exists a positive constant $c_{1,\beta}$ such that, for every $i,j\in\{1,\ldots,m\}$ and every $z\in \Sigma_{\beta}$, 
\begin{align*}
|p_{ij}(z,x,y)|\leq c_{1,\beta}({\rm Re}(z))^{-\frac d2}\exp\left(-\frac{c_2|x-y|^2}{8(1+\tan(\beta)){\rm Re}(z)}\sin\bigg (\frac{\theta-\beta}{2}\bigg )\right),
\end{align*}
for almost every $(x,y)\in\R^{d}\times\R^{d}$. Further, the same estimate holds true for $p_{ij}^* (i,j\in\{1,\ldots,m\})$.
\end{prop}

\begin{proof}
Fix $t>0$. From \eqref{stima_ker_smgr_gen} and straightforward computations, it follows that
\begin{align}
\label{stimaL2Linfinity_gen}
\|e^{-tL}\bm{f}\|_{\infty}
\le c_{2,\infty}t^{-\frac{d}{4}}\|\bm{f}\|_2
\end{align}
for every $\bm{f}\in C^\infty_c(\R^d;\C^m)$. Taking Remark \ref{rmk:adjoint} into account, we also deduce that 
\begin{align*}
\|e^{-tL^*}\bm{f}\|_{\infty}
\le c_{2,\infty}^*t^{-\frac{d}{4}}\|\bm{f}\|_2
\end{align*}
for every $\bm{f}\in C^\infty_c(\R^d;\C^m)$. In particular, this implies that 
\begin{align}
\label{stimaL1L2_gen}
\|e^{-tL}\|_{\mathscr L(L^1(\R^d;\C^m),L^2(\R^d;\C^m))}\leq c_{2,\infty}^*t^{-\frac{d}{4}}.
\end{align}

We fix $\beta$ as in the statement and set $\beta'=\frac{\beta+\theta}{2}$, $\beta''=\frac{3\theta+\beta}{4}$ and observe that, if $t+is\in \Sigma_{\beta'}$, then
$\delta t+is\in \Sigma_{\beta''}$ with $\delta=\frac{\tan(\beta')}{\tan(\beta'')}\in(0,1)$. 
Since  $(e^{-tL})_{t\ge 0}$ extends to a bounded analytic semigroup in $L^2(\R^d;\C^m)$ in the sector $\Sigma_\theta$, from \eqref{stimaL2Linfinity_gen} and \eqref{stimaL1L2_gen} we infer that
\begin{align}
\label{stime_smgr_complex_gen}
&\|e^{-(t+is)L}\|_{\mathscr L(L^1(\R^d;\C^m),L^\infty(\R^d;\C^m))} \notag \\
\leq & \|e^{-\frac{1-\delta}{2}tL}\|_{\mathscr L(L^1(\R^d;\C^m),L^2(\R^d;\C^m))}
\|e^{-(\delta t+is)L}\|_{\mathscr L(L^2(\R^d;\C^m),L^2(\R^d;\C^m))} \notag \\
& 
\quad\times\|e^{-\frac{1-\delta}{2}tL}\|_{\mathscr L(L^2(\R^d;\C^m),L^\infty(\R^d;\C^m))}\notag  \\
\leq & ct^{-\frac d2}
\end{align}
for every $t+is\in \Sigma_{\beta'}$, where, here and below, $c$ denotes a positive constant, which depends at most on $d$, $m$, $\beta$, $c_1$, and may vary from line to line.

Let us fix $i,j\in\{1,\ldots,m\}$, $f,g\in C^\infty_c(\R^d;\C)$ and two bounded subsets $E,F$ of $\R^d$ with positive distance. We introduce the function $G:\Sigma_{\beta'}\to\C$, defined by
\begin{equation}
\label{funct_G_gen}
G(z)\coloneqq\int_{\R^d}\langle (e^{-zL}(\chi_F f\bm{e}_j))(x),\chi_E(x)g(x)\bm{e}_i\rangle dx=\int_{\R^d}((e^{-zL}(\chi_F f\bm{e}_j))(x))_i\chi_E(x)\overline{g(x)}dx
\end{equation}
for every $z\in \Sigma_{\beta'}$. From \eqref{stima_ker_smgr_gen}, it follows that
\begin{align*}
|G(t)|\leq c_1t^{-\frac d2}e^{-\frac{c_2(d(E,F))^2}{t}}\|f\|_1\|g\|_1, \qquad t>0.   \end{align*}
Further, from \eqref{stime_smgr_complex_gen}, we infer that
\begin{align*}
|G(z)|\leq c({\rm Re}\,z)^{-\frac d2}\|f\|_1\|g\|_1, \qquad\;\, z\in \Sigma_{\beta'}.    
\end{align*}
By applying \cite[Lemma 6.18]{Ou05}, with $\alpha=1$ and $b=c_2(d(E,F))^2$,  we deduce that 
\begin{align}
\label{stima_G_1_gen}
|G(z)|\leq c({\rm Re}(z))^{-\frac d2}\exp\bigg (-\frac{c_2(d(E,F))^2}{2|z|}\sin\bigg (\frac{\theta-\beta}{2}\bigg )\bigg )\|f\|_1\|g\|_1
\end{align}
for every $z\in \Sigma_{\beta}$. The arbitrariness of the functions $f$ and $g$ implies that the operator $f\mapsto (e^{-zL}(\chi_Ff\bm{e}_j))_i\chi_E$ is bounded from $L^1(\R^d;\C)$ into $L^\infty(\R^d;\C)$ and
\begin{align*}
\|(e^{-zL}(\chi_F\cdot\bm{e}_j))_i\chi_E\|_{\mathscr L(L^1(\R^d;\C),L^\infty(\R^d;\C))}\leq c({\rm Re}(z))^{-\frac{d}{2}}\exp\bigg (\!-\frac{c_2(d(E,F))^2}{2|z|}\sin\bigg (\frac{\theta\!-\!\beta}{2}\bigg )\!\bigg )
\end{align*}
for every $z\in \Sigma_\beta$, which gives
\begin{align}
\label{stime_nucleo_compl_gen}
|p_{ij}(z,x,y)\chi_E(x)\chi_F(y)|\leq c({\rm Re}(z))^{-\frac d2}\exp\bigg (-\frac{c_2(d(E,F))^2}{2|z|}\sin\bigg (\frac{\theta-\beta}{2}\bigg )\bigg )  
\end{align}
for every $z\in \Sigma_{\beta}$, and almost every $(x,y)\in\R^{d}\times\R^{d}$ (see the proof of \cite[Theorem 1.3]{AB94}). 

Fix $z\in \Sigma_\beta$ and $(x,y)\in\R^{d}\times\R^{d}$ such
that \eqref{stime_nucleo_compl_gen} is satisfied. 
Further, set $E=B(x,r)$ and $F=B(y,r)$, where $r=\frac{1}{4}|x-y|$. Then, it follows that $d(E,F)=\frac{|x-y|}{2}$, which replaced in \eqref{stime_nucleo_compl_gen} yields the assertion observing that 
$|z|\le (1+{\rm tan}(\beta)){\rm Re}(z)$ for every $z\in \Sigma_{\beta}$.

Finally, the estimate for $p_{ij}^*$ ($i,j\in\{1,\ldots,m\}$) follows from Remark \ref{rmk:adjoint}.
\end{proof}

By means of Proposition \ref{prop:cplx_kernel_est_gen}, we provide Gaussian kernel estimates for the time-derivatives of the function $z\mapsto e^{-zL}$. To be more precise, we show kernel estimates for the operator $L^ke^{-zL}$, $k\in\N$ and $z\in\Sigma_{\beta}$ with $\beta\in\left(0,\theta\right)$, and, recalling that $\frac{d^k}{dz^k}e^{-zL}=(-L)^ke^{-zL}$ for every $z\in \Sigma_\theta$, to get the desired estimates. 

Fix $k\in \N$. From 
\eqref{stimaL2Linfinity_gen}, \eqref{stimaL1L2_gen} and the analyticity of $(e^{-tL})_{t\ge0}$ in $L^2(\R^d;\C^m)$, we deduce that
\begin{align*}
& \|L^k e^{-tL}\|_{\mathscr L(L^1(\R^d;\C^m),L^\infty(\R^d;\C^m))} \\
= & \|e^{-\frac t3L}L^k e^{-\frac t3L}e^{-\frac t3L}\|_{\mathscr L(L^1(\R^d;\C^m),L^\infty(\R^d;\C^m))} \\
\leq & \|e^{-\frac t3L}\|_{\mathscr L(L^1(\R^d;\C^m),L^2(\R^d;\C^m))}\|L^k e^{-\frac t3L}\|_{\mathscr L(L^2(\R^d;\C^m))}
\|e^{-\frac t3L}\|_{\mathscr L(L^2(\R^d;\C^m),L^\infty(\R^d;\C^m))} \\
\leq & ct^{-\frac d2-k}
\end{align*}
for every $t>0$ and some positive constant $c$.

Arguing componentwise, from the Dunford-Pettis Theorem we deduce that for every $i,j\in\{1,\ldots,m\}$ and every $t\in(0,\infty)$, there exists a kernel $p_{ij}^k(t,\cdot,\cdot)\in L^\infty(\R^d\times\R^d)$ such that
\begin{equation*}
(L^k e^{-tL}(\varphi \bm{e}_j)(x))_i
= \int_{\R^d}p_{ij}^k(t,x,y)\varphi(y)dy
\end{equation*}
for every $t>0$,  $\varphi\in L^1(\R^d)$, $i,j\in\{1,\ldots,m\}$ and almost every $x\in\R^d$. Further, $\|p_{ij}^k(t,\cdot,\cdot)\|_{\infty}\leq ct^{-\frac d2-k}$ for every $t>0$ and every $i,j\in\{1,\ldots,m\}$. 

The following result provides estimates of $p_{ij}^k$, $k\in\N$, also in terms of the spatial variables.

\begin{prop}\label{prop:stime:nucleo:He-tH}
For every $\beta\in\left(0,\theta\right)$ and $k\in\N$, there exists a positive constant $c_{1,k,\beta}$ such that, for every $i,j\in\{1,\ldots,m\}$ and every $z\in \Sigma_\beta$,
\begin{align*}
|p_{ij}^k(z,x,y)|\leq c_{1,k,\beta}({\rm Re}(z))^{-\frac d2-k}\exp\bigg (-\frac{c_2|x-y|^2}{16(1+\tan(\beta)){\rm Re}(z)}\sin\bigg (\frac{\theta-\beta}{4}\bigg )\bigg )
\end{align*}
for almost every $(x,y)\in\R^{d}\times\R^{d}$.
\end{prop}

\begin{proof}
Throughout the proof, $c$ will denote a positive constant, which depends at most on $d$, $m$, $\beta$, $\theta$ and $k$, but it is independent of the functions that we consider. Moreover, it may vary from line to line.

Fix $\beta$ as in the statement and set
$\beta'=\frac{\theta+\beta}{2}$ and
$\beta''=\frac{3\theta+\beta}{4}$. For every $i,j\in\{1,\ldots,m\}$, $f,g\in C^\infty_c(\R^d;\C)$, let $G$
be the function in \eqref{funct_G_gen}. Note that $G$ is holomorphic in $\Sigma_{\theta}$. Moreover, for every $z\in \Sigma_{\beta}$, the choice $r=\frac12\sin(\beta'-\beta){\rm Re}(z)$ implies that the circumference of centre $z$ and radius $r$ is contained in $\Sigma_{\beta'}$. From the Cauchy formula, we deduce that
\begin{align*}
G^{(k)}(z)
= \frac{k!}{2\pi i}\int_{\partial B(z,r)}\frac{G(\lambda)}{(\lambda-z)^{k+1}}d\lambda
= \frac{k!}{2\pi} \int_0^{2\pi}\frac{G(z+re^{i\tau})}{r^ke^{ik\tau}}d\tau, \qquad\;\, z\in \Sigma_{\beta}.
\end{align*}
If in the proof of Proposition \ref{prop:cplx_kernel_est_gen} we consider $\beta=\beta'$ and $\beta'=\beta''$, then from \eqref{stima_G_1_gen}, we get, for every $z\in \Sigma_{\beta}$,
\begin{align*}
|G^{(k)}(z)|
\leq 
cr^{-k}\|f\|_1\|g\|_1\int_0^{2\pi}({\rm Re}(z+re^{i\tau}))^{-\frac d2}\exp\bigg (-\frac{c_2(d(E,F))^2}{2|z+re^{i\tau}|}\sin\bigg (\frac{\theta-\beta}{4}\bigg )\bigg )d\tau.
\end{align*}
Since ${\rm Re}(z+re^{i\tau})\geq \left(1-\frac12\sin(\beta'-\beta)\right){\rm Re}(z)$
and $|z+re^{i\tau}|\leq |z|+r\leq 2|z|$, it follows that
\begin{align*}
|G^{(k)}(z)|
\leq c({\rm Re}(z))^{-\frac d2-k}\exp\bigg (-\frac{c_2(d(E,F))^2}{4|z|}\sin\bigg (\frac{\theta-\beta}{4}\bigg )\bigg )\|f\|_1\|g\|_1, \qquad z\in \Sigma_{\beta}.
\end{align*}
The assertion now follows arguing as in the final part of the proof of Proposition \ref{prop:cplx_kernel_est_gen}.
\end{proof}

Now we are able to prove the following result.
\begin{theo}
\label{thm:Hinfinity_calculus_Lp:gen}
The operator $L$ admits a bounded $\textup{H}^\infty(\Sigma_\mu)$-calculus on $L^p(\R^d;\C^m)$ for every $p\in(1,\infty)$ and every $\mu\in\left(\frac\pi2-\theta,\pi\right)$.
\end{theo}
\begin{proof}

Throughout the proof, $c$ denotes a positive constant, which may vary from line to line but is independent of the functions that we consider, of $x\in\R^d$ and of $\alpha\in (0,\infty)$. Moreover, for every $\psi\in [0,\pi]$, we denote by $\gamma_\psi^{\pm}$ the curve in the complex plane defined by
$\gamma_\psi^{\pm}(s)=
se^{\pm i\psi}$ for every $s\in (0,\infty)$.

We split the proof into several steps. First, we prove that the operator $\xi(L)$ is of weak type $(1,1)$ when $\xi\in \Psi(\Sigma_\mu)$ for some $\mu\in \left(\frac\pi2-\theta,\pi\right)$. This is obtained in Steps 1 to 4.
Next in Step 5, we remove the extra condition on the behaviour of $\xi$ at the origin and at infinity, and complete the proof. 

{\em Step 1}. Let us fix $\xi\in\Psi(\Sigma_{\mu})$ and consider the operator $\xi(L)\in {\mathscr L}(L^2(\R^d;\C^m))$, which is defined by
\begin{align*}
\xi(L)=\frac{1}{2\pi i}\int_{\gamma_\psi^+}\xi(\lambda)R(\lambda,L)d\lambda-\frac{1}{2\pi i}\int_{\gamma_\psi^-}\xi(\lambda)R(\lambda,L)d\lambda,    
\end{align*}
where $\psi$ is arbitrarily fixed in $\left(\frac\pi2-\theta,\mu\right)$. 
We further fix ${\bm{u}}\in C^\infty_c(\R^d;\C^m)$, split it into the sum $\bm{u}=\sum_{k=1}^mu_k\bm{e}_k$ and note that 
\begin{align}
\label{spezzamento_1}
|\{x\in \R^d:\|(\xi(L)\bm{u})(x)\|>\alpha\}|
\leq & \bigg|\bigg\{x\in \R^d:\sum_{k=1}^m\|(\xi(L)(u_k\bm{e}_k))(x)\|>\alpha\bigg\}\bigg | \notag \\
\leq & \sum_{k=1}^m \left|\left\{x\in \R^d:\|(\xi(L)(u_k\bm{e}_k))(x)\|>\frac\alpha m\right\}\right|
\end{align}
for every $\alpha>0$.

Now, we fix $k\in\{1,\ldots,m\}$ and consider the Calder\'on-Zygmund decomposition of $u_k$  at height $\alpha>0$ (see e.g., \cite[Corollary 2.3]{CoWe71}), i.e., we consider functions $g$, $f_i$ and balls $B(x_i,r_i)\subset\R^d$ ($i\in\N$) such that, for some positive constant $c$, which only depends on the dimension $d$,
\begin{enumerate}[\rm(i)]
\item 
$u_k=g+h$, with $h=\sum_{i\in\N}f_i$; 
\item 
$|g(x)|\leq c\alpha$ for every $x\in \R^d$;
\item 
for every $i\in\N$, the function $f_i$ has support contained in $B(x_i,r_i)$;
\item there exists $N\in\N$ such that each $x\in \R^d$ is contained in at most $N$ balls of $\{B(x_i,r_i): i\in\N\}$;
\item 
$\|f_i\|_1\leq c\alpha r_i^d$ for every $i\in\N$;
\item 
$\sum_{i\in\N}r_i^d\leq c\alpha^{-1}\|u_k\|_1$.
\end{enumerate}

Clearly, $h$ and $g$ can be chosen with compact support and, from (v) and (vi), it follows that $\sum_{i\in\N}\|f_i\|_1<\infty$, which implies that
the function $h$ belongs to $L^1(\R^d;\C)$ and $\|h\|_1\leq c\|u_k\|_1$ and, in particular, $\|g\|_1\leq c\|u_k\|_1$. 

Let us split 
\begin{align*}
u_k\bm{e}_k
=g\bm{e}_k+\sum_{i\in\N}f_i\bm{e}_k=
 g\bm{e}_k+\sum_{i=1}^{\infty}e^{-r_i^2L}(f_i\bm{e}_k)+\sum_{i=1}^{\infty}(I-e^{-r_i^2L})(f_i\bm{e}_k)\eqqcolon\bm {g}+\bm{h}_1+\bm{h}_2
\end{align*}
and observe that $\bm{g}$ has compact support. Moreover,  $\bm{h}_1$ (and, by difference, $\bm{h}_2$) belongs to $L^2(\R^d;\C^m)$. We give it for granted for the moment and
prove the claim in Step 2.

The above decomposition implies that
\begin{align}
\label{spezzamento_2}
\left|\left\{x\in \R^d:\|(\xi(L)(u_k\bm{e}_k))(x)\|>\frac\alpha m\right\}\right|
\leq & \left|\left\{x\in \R^d:\|(\xi(L)\bm{g})(x)\|>\frac\alpha{3 m}\right\}\right| \notag \\
& + \left|\left\{x\in \R^d:\|(\xi(L)\bm{h}_1)(x)\|>\frac\alpha{3 m}\right\}\right| \notag \\
& + \left|\left\{x\in \R^d:\|(\xi(L)\bm{h}_2)(x)\|>\frac\alpha{3 m}\right\}\right|\notag\\
\eqqcolon & {\mathscr M}_{{\bm g}}+{\mathscr M}_{{\bm h}_1}+
{\mathscr M}_{{\bm h}_2}.
\end{align}
We now estimate the three terms in the last side of \eqref{spezzamento_2} separately.

{\em Step 2 (Estimate of $\mathscr{M}_{\bm{g}}$)}. Let us first assume  that $\|\xi\|_{\infty,\mu}=1$. From Chebyshev's inequality, the fact that $\|\xi(L)\|_{\mathscr L(L^2(\R^d;\C^m))}<\infty$ and (ii), we infer that
\begin{align*}
\mathscr{M}_{\bm{g}}
\leq & \frac{c}{\alpha^{2}}\int_{\R^d}\|(\xi(L)\bm{g})(x)\|^2dx 
\leq  \frac{c}{\alpha^{2}}\int_{\R^d}\|\bm{g}(x)\|^2dx
\leq \frac{c}{\alpha}\int_{\R^d}\|\bm{g}(x)\|dx \leq  \frac{c}{\alpha}\|u_k\|_1.
\end{align*}
In the general case, setting $\tilde \xi=\xi\|\xi\|_{\infty,\mu}^{-1}$, we can write
\begin{align}
\label{stima_g_finale}
\mathscr{M}_{\bm{g}}
= & \bigg |\bigg \{x\in \R^d:\|(\tilde\xi(L)\bm{g})(x)\|>\frac\alpha{3m\|\xi\|_{\infty,\mu}}\bigg \}\bigg |
\leq  \frac{c}{\alpha}\|\xi\|_{\infty,\mu}\|u_k\|_1.   \end{align}

{\em Step 3 (Estimate of $\mathscr{M}_{\bm{h}_1}$)}. Also in this case, we first consider the case when $\|\xi\|_{\infty,\mu}=1$. The same arguments as in Step 2, show that
\begin{align}
\label{stima_h1}
{\mathscr M}_{\bm{h}_1}
\leq \frac{c}{\alpha^2}\|\bm{h}_1\|_2^2.
\end{align}
From \eqref{stima_ker_smgr_gen}, Lemma \ref{lemma:kernel_est}(iii) (with $v=1$), the expression of $w$ (see \eqref{eq:def:gaussian:fun}) and properties (iii) and (v) in the Calder\'on-Zygmund decomposition of $u_k$, we infer that, for every $j\in\{1,\ldots,m\}$ and almost every $x\in\R^d$, there exist positive constants $c$ and $\tau$, independent of $i$, such that
\begin{align*}
|((e^{-r_i^2L}(f_i\bm{e}_k))(x))_j|
\leq & \int_{\R^d}|p_{jk}(r_i^2,x,y)||f_i(y)|dy 
\leq c_1r_i^{-d}\int_{\R^d}e^{-\frac{c_2|x-y|^2}{r_i^2}}|f_i(y)|dy \\
\leq & cr_i^{-d}\|f_i\|_1\sup_{y\in B(x_i,r_i)}e^{-\frac{c_2|x-y|^2}{r_i^2}} 
\leq  c\alpha(\tau r_i^2)^{-\frac{d}{2}}r_i^d\inf_{y\in B(x_i,r_i)}e^{-\frac{c_2|x-y|^2}{\tau r_i^2}} \\
\leq & c\alpha (\tau r_i^2)^{-\frac{d}{2}}\int_{\R^d}e^{-\frac{c_2|x-y|^2}{\tau r_i^2}}\chi_{B(x_i,r_i)}(y)dy.
\end{align*}
From Lemma \ref{lemma:kernel_est}(ii) we deduce that
\begin{align*}
\bigg |\int_{\R^d}\langle (e^{-r_i^2L}(f_i\bm{e}_k))(x),\bm{\varphi}(x)\rangle dx\bigg |
\leq & \sum_{j=1}^m\int_{\R^d} |((e^{-r_i^2L}(f_i\bm{e}_k))(x))_j||\varphi_j(x)|dx \\
\leq & c\alpha (\tau r_i^2)^{-\frac d2}\sum_{j=1}^m\int_{\R^d}dx\int_{\R^d}e^{-\frac{c_2|x-y|^2}{\tau r_i^2}}\chi_{B(x_i,r_i)}(y)|\varphi_j(x)|dy \\
\leq & c\alpha\sum_{j=1}^m\int_{\R^d}(M|\varphi_j|)(y)\chi_{B(x_i,r_i)}(y)dy
\end{align*}
for every $\bm{\varphi}\in L^2(\R^d;\C^m)$. By taking properties (iv) and (vi) in the above decomposition and the boundedness of the maximal function $M$ on $L^2(\R^d;\C)$ into account, it follows that
\begin{align*}
\|\bm{h}_1\|_2
= & \sup_{\|\bm{\varphi}\|_{L^2(\R^d;\C^m)}\leq1}\bigg |\int_{\R^d}\bigg\langle\sum_{i\in\N}(e^{-r_i^2 L}(f_i\bm{e}_k))(x),\bm{\varphi}(x)\bigg\rangle dx\bigg | \\
\leq & c\alpha\sup_{\|\bm{\varphi}\|_{L^2(\R^d;\C^m)}\leq1}\sum_{j=1}^m\int_{\R^d}\sum_{i\in\N}\chi_{B(x_i,r_i)}(x)(M|\varphi_j|)(x)dx \\
\leq & c\alpha \sup_{\|\bm{\varphi}\|_{L^2(\R^d;\C^m)}\leq1}\sum_{j=1}^m\|M|\varphi_j|\|_2\bigg\|\sum_{i\in\N}\chi_{B(x_i,r_i)}\bigg\|_2 \\
\leq & c\alpha\bigg(\sum_{i\in\N}|B(x_i,r_i)|\bigg )^{\frac12}
\leq c\alpha^{\frac12}\|u_k\|_1^{\frac12}.
\end{align*}
This estimate allows us to control the right-hand side of
\eqref{stima_h1} with $c\alpha^{-1}\|u_k\|_1$. Then, arguing as in \eqref{stima_g_finale}, we infer that
\begin{align}
\label{stima_h1_finale}
\mathscr{M}_{\bm{h}_1}
\leq \frac{c}{\alpha}\|\xi\|_{\infty,\mu}\|u_k\|_1
\end{align}
for a general $\xi\in\Psi(\Sigma_{\mu})$.

{\em Step 4 (Estimate of $\mathscr{M}_{\bm{h}_2}$)}. It remains to estimate the last term in the right-hand side of \eqref{spezzamento_2}. For this purpose, we observe that
\begin{align}
\mathscr{M}_{\bm{h}_2}
\leq & \sum_{i=1}^{\infty}|B(x_i,2r_i)|+\bigg|\bigg\{x\in \R^d\setminus\bigcup_{j\in\N}B(x_j,2r_j): \|(\xi(L)\bm{h}_2)(x)\|_1>\frac\alpha{3 m}\bigg\}\bigg|\notag \\
\leq & \frac{c}{\alpha}\|u_k\|_1
+\sum_{\ell=1}^m \bigg|\bigg\{x\in \R^d\setminus\bigcup_{j\in\N}B(x_j,2r_j):|((\xi(L)\bm{h}_2)(x))_{\ell}|>\frac\alpha{3 m^{3/2}}\bigg\}\bigg|,
\label{stima_h2-iniziale}
\end{align}
where we have used the doubling property of the Lebesgue measure and property (vi) in the above decomposition of $u_k$. 

Let us estimate the second term in the last side of \eqref{stima_h2-iniziale}. To this aim, we recall that, for every $\eta\in\left(0,\theta\right)$, $(e^{-tL})_{t\geq0}$ extends to an analytic semigroup in the sector $\Sigma_\eta$.  We set 
\begin{align*}
\omega\coloneqq
\begin{cases}
\frac{\pi+\theta}{2}, & \mu\in\left[\frac\pi2+\theta,\pi\right), \\[1mm]
\frac{3\pi}{4}+\frac{\theta-\mu}{2}, & \mu\in\left(\frac\pi2-\theta,\frac\pi2+\theta\right),
\end{cases}
\end{align*}
and $\xi_j(z)\coloneqq\xi(z)(1-e^{-r_j^2z})$ for every $z\in \Sigma_{\pi-\omega}$ and every $j\in\N$. Since $\pi-\omega\in (\frac\pi2-\theta,\mu)$,  it follows that $\xi_j$ belongs to $\Psi(\Sigma_{\pi-\omega})$ and
\begin{align}
\xi(L)\bm{h}_2
= & \sum_{j=1}^{\infty}\xi(L)(I-e^{-r_j^2L})(f_j\bm{e}_k) \notag \\
= & -\frac{1}{2\pi i}\sum_{j=1}^{\infty}\bigg (\int_{\gamma_{\omega}^+}\xi_j(-\lambda)R(\lambda,-L)d\lambda\bigg )(f_j\bm{e}_k)\notag\\
&+\frac{1}{2\pi i}\sum_{j=1}^{\infty}\bigg (\int_{\gamma_{\omega}^-}\xi_j(-\lambda)R(\lambda,-L)d\lambda\bigg )(f_j\bm{e}_k)\notag\\
=&\sum_{j=1}^{\infty}\xi_j(L)(f_j\bm{e}_k).
\label{rappresentazione_xi_jA}
\end{align}

We claim that, if we set
\begin{align*}
\eta\coloneqq
\begin{cases}
\frac23\theta, & \mu\in\left[\frac\pi2+\theta,\pi\right), \\[1mm]
\frac{\pi}{8}+\frac{3\theta-\mu}{4}, & \mu\in\left(\frac\pi2-\theta,\frac\pi2+\theta\right),
\end{cases}
\end{align*}
then
\begin{align}
\label{rappresentazione_risolvente}
R(\lambda,-L)=\int_{\gamma_{\eta}^-} e^{-\lambda z}e^{-z L}dz, \quad\;\,\lambda\in\C \textit{ such that }\, {\rm Re}(\lambda)<0,\,{\rm arg}(\lambda)=\omega.    
\end{align}
We note that the function $\zeta\mapsto \int_{\gamma_{\eta}^-} e^{-\zeta z}e^{-z L}dz$ is analytic in the sector $\{\zeta\in\C: |{\rm arg}(\zeta)-\eta|<\frac{\pi}{2}\}$. Moreover, when $\zeta$ is taken in $(0,\infty)$, a simple computation shows that
\begin{equation*} 
\int_{\gamma_{\eta}^-} e^{-\zeta z}e^{-z L}dz=\int_0^{\infty}e^{-\zeta z}e^{-z L}dz.
\end{equation*}
By analyticity, formula \eqref{rappresentazione_risolvente} follows immediately.
In the same way we can show that
\begin{align}
R(\lambda,-L)=\int_{\gamma_{\eta}^+} e^{-\lambda z}e^{-z L}dz, \quad\;\,\lambda\in\C \textit{ such that }\, {\rm Re}(\lambda)<0,\,  {\rm arg}(\lambda)=-\omega. \label{rappresentazione_risolvente-1}
\end{align}

From \eqref{rappresentazione_xi_jA}, \eqref{rappresentazione_risolvente} and \eqref{rappresentazione_risolvente-1}, it follows that, for every $j\in\N$,
\begin{align*}
&\frac{1}{2\pi i}\int_{\gamma_{\omega}^{\pm}}((\xi_j(-\lambda)R(\lambda,-L)(f_j\bm{e}_k))(x))_{\ell}d\lambda\\
=& \frac{1}{2\pi i}
\int_{\gamma_{\eta}^{\mp}}dz\int_{\gamma_{\omega}^{\pm}}e^{-\lambda z}\xi_j(-\lambda) ((e^{-z L}(f_j\bm{e}_k))(x))_{\ell}d\lambda
\end{align*}
for every $\ell\in\{1,\ldots,m\}$ and almost every $x\in \R^d\setminus B(x_j,2r_j)$. Hence, 
\begin{align}
&{\mathscr I}_{j,\ell}^{\pm}\coloneqq \int_{\R^d\setminus B(x_j,2r_j)}\bigg|\frac{1}{2\pi i}\int_{\gamma_{\omega}^{\pm}}((\xi_j(-\lambda)R(\lambda,-L)(f_j\bm{e}_k))(x))_{\ell}d\lambda\bigg|
dx \label{stima_integrale_compl}\\ 
 \leq & c\|\xi\|_{\infty,\mu}\int_0^\infty d\sigma\int_0^\infty e^{-s\sigma\Theta}|1-e^{r_j^2se^{\pm i\omega}}|ds\int_{\R^d\setminus B(x_j,2r_j)}|((e^{-\sigma e^{\mp i\eta}L}(f_j\bm{e}_k))(x))_{\ell}|dx,
\notag
\end{align}
where, to ease the notation, we set $\Theta\coloneqq\cos(\omega-\eta)$, which is positive by the definition of $\omega$ and $\eta$.
Since the support of $f_j$ is contained in $B(x_j,r_j)$, from Lemma \ref{lemma:kernel_est}(i), with $\gamma=2$, and Proposition \ref{prop:cplx_kernel_est_gen}, with $\beta=\frac{\eta+\theta}{2}\in(\eta,\theta)$, we infer that
\begin{align*}
&\int_{\R^d\setminus B(x_j,2r_j)}|((e^{-\sigma e^{\mp i\eta}L}(f_j\bm{e}_k))(x))_{\ell}|dx\\
\leq & \int_{B(x_j,r_j)}|f_j(y)|dy\int_{\R^d\setminus B(y,r_j)}|p_{\ell k}(\sigma e^{\mp i\eta},x,y)|dx\\
\leq & c(\sigma\cos(\eta))^{-\frac{d}{2}}\int_{B(x_j,r_j)}|f_j(y)|dy\int_{\R^d\setminus B(y,r_j)}e^{-c\frac{|x-y|^2}{\sigma\cos(\eta)}}dx\\
\leq & c\frac{\sigma^2}{(\sigma+r_j^2)^2}\|f_j\|_1.§
\end{align*}
Replacing in \eqref{stima_integrale_compl}, we infer that
\begin{align*}
{\mathscr I}_{j,\ell}^{\pm}\leq  c\|\xi\|_{\infty,\mu}\|f_j\|_1\int_0^\infty
\frac{\sigma^2}{(\sigma+r_j^2)^2}d\sigma\int_0^\infty e^{-s\sigma \Theta}|1-e^{r_j^2se^{\mp i\omega}}|ds.
\end{align*}
Let us notice that $|1-e^{r_j^2se^{\mp i\omega}}|\leq 2$ for every $s>0$ and $|1-e^{r_j^2se^{\mp i\omega}}|\leq sr_j^2$ if $sr_j^2\le 1$. Taking this remark into account, we can go further in the estimate of the term ${\mathscr I}^{\pm}_{j,\ell}$ and write 
\begin{align*}
{\mathscr I}_{j,\ell}^{\pm}\leq & c\|\xi\|_{\infty,\mu}\|f_j\|_1\bigg (\int_0^\infty\! \frac{\sigma^2}{(\sigma+r_j^2 )^2}d\sigma\!\int_{r_j^{-2}}^\infty e^{-s\sigma\Theta}ds\hskip -1pt+\hskip -1pt \int_0^\infty\! \frac{\sigma^2}{(\sigma+r_j^2)^2}d\sigma\int_0^{r_j^{-2}} r_j^2se^{-s\sigma \Theta}ds\bigg) \\
= &c\Theta^{-2}\|\xi\|_{\infty,\mu}\|f_j\|_1
\int_0^{\infty}
\frac{r_j^2}{(\sigma+r_j^2)^2}(1-e^{-\sigma r_j^{-2}\Theta })d\sigma\\
=&c\Theta^{-2}\|\xi\|_{\infty,\mu}\|f_j\|_1\int_0^{\infty}
\frac{1}{(1+\sigma)^2}(1-e^{-\sigma\Theta })d\sigma
\end{align*}
for every $\ell\in\{1,\ldots,m\}$. 
Hence, we conclude first that $\mathscr I_{j,\ell}^{\pm}\leq c\|\xi\|_{\infty,\mu}\|f_j\|_1$
for some constant $c$, which is independent of $j$, $k$ and $\ell$, and then, from \eqref{rappresentazione_xi_jA}, that
\begin{equation}
\label{stima_h2_2}
\int_{\R^d\setminus B(x_j,2r_j)}|((\xi_j(L)(f_j\bm{e}_k)(x))_{\ell}|dx\leq c\|\xi\|_{\infty,\mu}\|f_j\|_{1}, \qquad j\in\N, \;\, \ell\in\{1,\ldots,m\}   
\end{equation}
for some constant $c$, which is independent of $j$ and $k$.

From the properties (v), (vi) above and \eqref{stima_h2_2}, we conclude that 
\begin{align}
& \sum_{\ell=1}^m \bigg|\bigg\{x\in \R^d\setminus\bigcup_{i\in\N}B(x_i,2r_i):|((\xi(L)\bm{h}_2)(x))_{\ell}|>\frac\alpha{3 m^{3/2}}\bigg\}\bigg| \notag \\
\leq & \frac{c}{\alpha}\sum_{\ell=1}^m\int_{\R^d\setminus \bigcup_{i\in\N}B(x_i,2r_i)}|((\xi(L)\bm{h}_2)(x))_{\ell}|dx \notag \\
\leq & \frac{c}{\alpha}\sum_{\ell=1}^m\int_{\R^d\setminus\bigcup_{i\in\N} B(x_i,2r_i)}\bigg |\sum_{j=1}^{\infty}((\xi_j(L)(f_j\bm{e}_k))(x))_{\ell}\bigg |dx \notag \\
\leq & \frac{c}{\alpha}\sum_{\ell=1}^m\sum_{j=1}^{\infty}\int_{\R^d\setminus B(x_j,2r_j)}|((\xi_j(L)(f_j\bm{e}_k))(x))_{\ell}|dx \notag \\
\leq &  \frac{c}{\alpha}\|\xi\|_{\infty,\mu}\sum_{j=1}^{\infty}\|f_j\|_1 \notag \le\frac{c}{\alpha}\|\xi\|_{\infty,\mu}\|u_k\|_1.
\end{align}
By assuming that $\|\xi\|_{\infty,\mu}=1$, from \eqref{stima_h2-iniziale} we obtain that
\begin{align*}
{\mathscr M}_{\bm{h}_2}\le 
\frac{c}{\alpha}
\|u_k\|_1
\end{align*}
and, arguing as for the estimates of ${\mathscr M}_{\bm{g}}$ and ${\mathscr M}_{\bm{h}_1}$, we infer that
\begin{align}
\label{stima_h2_finale:gen}
\mathscr{M}_{\bm{h}_2}
\leq \frac{c}{\alpha}\|\xi\|_{\infty,\mu}\|u_k\|_1
\end{align}
for every $\xi\in\Psi(\Sigma_{\mu})$, where $c$ is a positive constant which only depends on $d$, $m$, $\mu$ and $\theta$ (in particular, it does not depend on $\xi$, $\alpha$ and $\uu$).

{\em Step 5}. Replacing \eqref{stima_g_finale}, \eqref{stima_h1_finale} and \eqref{stima_h2_finale:gen} in \eqref{spezzamento_2}, we infer that 
\begin{align*}
\left|\left\{x\in \R^d:\|(\xi(L)(u_k\bm{e}_k))(x)\|>\frac\alpha m\right\}\right|
\leq\frac{c}{\alpha}\|\xi\|_{\infty,\mu}\|u_k\|_1, \qquad k\in\{1,\ldots,m\},
\end{align*}
which, together with \eqref{spezzamento_1}, yields to 
\begin{align}
\label{stima_finale}
|\{x\in \R^d:\|(\xi(L)\bm{u})(x)\|>\alpha\}|
\leq \frac{c}{\alpha}\|\xi\|_{\infty,\mu}\|\bm{u}\|_1
\end{align}
for every $\xi\in\Psi(\Sigma_{\mu})$,  $\uu\in C_c^\infty(\R^d;\C^m)$ and $\alpha>0$, where $c$ is a positive constant which only depends on $d$, $m$, $\mu$ and $\theta$.

Next, we extend \eqref{stima_finale} to every function in $\textup{H}^\infty(\Sigma_\mu)$. For this purpose, we fix $\xi\in \textup{H}^\infty(\Sigma_\mu)$ and consider the sequence $(\xi_n)\subset \Psi(\Sigma_\mu)$ defined as
\begin{align*}
\xi_n(z)\coloneqq\xi(z)\frac{z^\frac{1}{n}}{(1+z)^{\frac{2}{n}}}, \qquad z\in \Sigma_\mu, \;\, n\in\N.    
\end{align*}
From \cite[Theorem D]{AlDuMcIn96}, it follows that the sequence $(\xi_n(L)\bm{u})$ converges to $\xi(L)\bm{u}$ in $L^2(\R^d;\C^m)$ for every $\bm{u}\in L^2(\R^d;\C^m)$. Hence, there exists a subsequence $(\xi_{n_k}(L)\bm{u})$ which converges to $\xi(L)\bm{u}$ almost everywhere in $\R^d$. Writing \eqref{stima_finale} with $\xi$ being replaced by $\xi_{n_k}$ and observing that $\|\xi_{n_k}\|_{\infty,\mu}\le c\|\xi\|_{\infty,\mu}$ for every $k\in\N$ and some positive constant $c$, independent of $\xi$ and $k$, we deduce that
\begin{align}
|\{x\in \R^d: \|(\xi(L)\bm{u})(x)\|>\alpha\}|\le & \lim_{k\to\infty}\bigg |
\bigcap_{h\ge k}\{x\in\R^d: \|(\xi_{n_h}(L)\bm{u})(x)\|>\alpha\}\bigg |\notag\\
\le &\frac{c}{\alpha}\lim_{k\to\infty}\|\xi_{n_k}\|_{\infty,\mu}\|\bm{u}\|_1\notag\\
\leq& \frac{c}{\alpha}\|\xi\|_{\infty,\mu}\|\bm{u}\|_1
\label{stima_finale_Hinf_L2}
\end{align}
for every $\alpha>0$.
Choosing $\mathcal D=C^\infty_c(\R^d;\C^m)$ and $S=\xi(L)$ in Lemma \ref{lemma:approssimazione}, we conclude that $\xi(L)$ extends to an operator on $L^1(\R^d;\C^m)$ of weak type $(1,1)$.

Now we fix $j,k\in\{1,\ldots,m\}$. Since $|((\xi(L)(f\bm{e}_k))(x))_j|\leq \|(\xi(L)(f\bm{e}_k))(x)\|$ for every 
$f\in L^1(\R^d;\C)$ and
almost every $x\in\R^d$, from \eqref{stima_finale_Hinf_L2} it follows that 
\begin{align*}
|\{x\in \R^d:|((\xi(L)(f\bm{e}_k))(x))_j|>\alpha\}|
\leq \frac{c}{\alpha}\|\xi\|_{\infty,\mu}\|f\|_1, \qquad\;\, \alpha>0,\;\,j,k\in\{1,\ldots,m\}.    
\end{align*}
Hence, the operator $L_{j,k}$, defined by $L_{j,k}(f) \coloneqq (\xi(L)(f\bm{e}_k))_j$ for every $f\in L^1(\R^d;\C)$, is of weak type $(1,1)$ and it also belongs to $\mathscr L(L^2(\R^d;\C))$ since
$\xi(L)\in \mathscr L(L^2(\R^d;\C^m))$.
From the Marcinkiewicz interpolation theorem, it follows that $L_{j,k}\in \mathscr L(L^p(\R^d;\C))$ for every $p\in(1,2)$ and
\begin{align*}
\|\xi(L)\bm{u}\|_p
\leq & \sum_{k=1}^m\|\xi(L)(u_k\bm{e}_k)\|_p
\leq \sum_{k=1}^m\bigg(\sum_{j=1}^m\|(\xi(L)(u_k\bm{e}_k))_j\|_p^p\bigg )^{\frac1p} \\
\leq & c\|\xi\|_{\infty,\mu}\sum_{k=1}^m\|u_k\|_p
\leq  c\|\xi\|_{\infty,\mu}\|\bm{u}\|_p
\end{align*}
for every $\bm{u}\in C^\infty_c(\R^d;\C^m)$ and some positive constant $c$, depending on $p$.
By density, we conclude that $\xi(L)\in\mathscr L(L^p(\R^d;\C^m))$ and $\|\xi(L)\|_{\mathscr L(L^p(\R^d;\C^m))}\leq c\|\xi\|_{\infty,\mu}$.

Since the function $\widetilde\xi$, defined by
$\widetilde \xi(\lambda)=\overline{\xi(\overline\lambda)}$ for every $\lambda\in \Sigma_{\mu}$ belongs to
$\textup{H}^{\infty}(\Sigma_{\mu})$ and $(\xi(L))^*=\widetilde \xi(L^*)$, from Remark \ref{rmk:adjoint} and the previous steps we infer that
$(\xi(L))^*\in \mathscr L(L^p(\R^d;\C^m))$ for every $p\in(1,2)$. 

If we now take $p>2$, then $p'\in (1,2)$ and
\begin{align*}
\left|\int_{\R^d}\langle (\xi(L)\bm{f})(x),\bm{g}(x)\rangle dx \right|
= & \left|\int_{\R^d}\langle \bm{f}(x),(\xi(L))^*\bm{g})(x)\rangle dx \right|
\leq  c\|\xi\|_{\infty,\mu}\|\bm{f}\|_p\|\bm{g}\|_{p'}
\end{align*}
for every $\bm{f}, \bm{g}\in C^\infty_c(\R^d;\C^m)$. By density, we conclude that $\xi(L)\in\mathscr L(L^p(\R^d;\C^m))$ and $\|\xi(L)\|_{\mathscr L(L^p(\R^d;\C^m))}\leq c\|\xi\|_{\infty,\mu}$.
\end{proof}

\section{Schr\"odinger type operators}
We are going to apply Theorem \ref{thm:Hinfinity_calculus_Lp:gen} to show the boundedness of the $\textup{H}^\infty$-calculus on $L^p(\R^d;\C^m)$ for the elliptic operator $\mathscr{L}$ acting on vector-valued smooth functions $\bm{u}$ as
\begin{equation*}
\mathscr{L}\bm{u}=-\operatorname{div}(Q\nabla \bm{u})+V\bm{u},  
\end{equation*}
where the coefficients satisfy the following conditions.

\begin{hyps}\label{hyp_0}
$Q=(q_{ij})_{1\le i,j\le d}$ and $V=(v_{hk})_{1\le h,k\le m}$ are real matrix-valued functions such that
\begin{enumerate}[\rm (i)]
\item
$q_{ij}=q_{ji}\in L^\infty(\R^d;\R)$ for each $i,j\in\{1,\ldots,d\}$, and there exists a positive constant $\nu$ such that 
\begin{equation*}
\langle Q(x)\xi,\xi\rangle \ge \nu |\xi|^2,\qquad\;\,\xi\in \R^d,
\end{equation*}
for almost every $x\in\R^d$;
\item
$v_{hk}=v_{kh}\in L^1_{\rm loc}(\R^d;\R)$ for all $h,k\in\{1,\ldots,m\}$ and
\begin{equation}\label{eq:accretivity:V}
    \langle V(x)\zeta ,\zeta\rangle \ge 0,\qquad\;\,\zeta\in \R^m,
\end{equation}
for almost every $x\in\R^d$.
\end{enumerate}
\end{hyps}
By \cite{Maichine19} it follows that $\mathscr{L}$ admits a realization $L$ on $L^2(\R^d;\C^m)$ which is associated to the bilinear form
\begin{equation*}
\mathfrak{a}(\bm{u},\bm{v})=\sum_{i=1}^m\int_{\R^d}\langle Q(x)\nabla u_i(x),\nabla v_i(x)\rangle dx+\int_{\R^d}\langle V(x) \bm{u}(x),\bm{v}(x) \rangle dx, \qquad\;\, \bm{u},\bm{v}\in D(\mathfrak{a}),    
\end{equation*}
where
\begin{equation*}
  D(\mathfrak{a})\coloneqq\{\bm{f}\in H^1(\R^d;\C^m): V^{\frac{1}{2}}\bm{f}\in L^2(\R^d;\C^m)\} 
\end{equation*}
is endowed with the norm $\|\uu\|_{D(\mathfrak{a})}^2
=\|\uu\|_{L^2(\R^d;\C^m)}^2+\||\nabla\uu|\|_{L^2(\R^d;\C^m)}^2+\|V^{1/2}\uu\|_{L^2(\R^d;\C^m)}^2$.

Since $\mathfrak{a}$ is a densely defined, accretive, continuous and closed bilinear form, the operator $-L$ generates a $C_0$-semigroup of contractions $(e^{-tL})_{t\ge 0}$ on $L^2(\R^d;\C^m)$ which is bounded and analytic in the open right-halfplane of $\C$. This semigroup extends to a strongly continuous and analytic semigroup on $L^p(\R^d;\C^m)$ for every $p\in(1,\infty)$ (see, \cite[Theorem 3.5]{Maichine19}). Note that $(e^{-tL})_{t\ge 0}$ is not necessary a positive semigroup. It is positive if and only if the off-diagonal entries of $V$ are nonpositive, see \cite[Proposition 4.4]{KLMR19}.

In particular, since $\mathfrak{a}$ is accretive and symmetric, the self-adjoint operator $L$ is sectorial of angle $0$ and admits a bounded $\textup{H}^\infty(\Sigma_\mu)$-calculus in $L^2(\R^d;\C^m)$ for every $\mu\in\left(0,\pi\right)$.

In order to show that $(e^{-tL})_{t\ge 0}$ has a matrix kernel, which satisfies Gaussian estimates, we need a domination argument that can be obtained by the abstract result \cite[Theorem 2.30]{Ou05}, which applies provided we establish the following lemma. 

\begin{lemma}\label{lem:dominazione}
The domain $D(\mathfrak{a})$ of the form $\mathfrak{a}$, associated to the operator $L$, satisfies the following properties:
\begin{enumerate}[\rm(a)]
\item
${\uu}\in D(\mathfrak{a})$ implies $\|\uu\|\in H^1(\R^d;\C)$,
\item
${\uu}\in D(\mathfrak{a}),\ f\in H^1(\R^d;\C)$ with $|f|\le\|\uu\|$ implies $|f|{\rm sign}(\uu)\in D(\mathfrak{a})$.
\end{enumerate}
Moreover, for every $(\uu,f)\in D(\mathfrak{a})\times H^1(\R^d;\C)$ such that $|f|\le\|\uu\|$ it holds that
\begin{equation}
\label{eq:form:inequality}
\mathfrak{a}({\uu},|f|\,{\rm sign}(\uu)) \ge \mathfrak{b}(\|\uu\|,|f|),
\end{equation}
where ${\rm sign}(\uu) = \dfrac{\uu}{\|\uu\|}\chi_{\{\uu\ne \bm{0}\}}$ and $\mathfrak{b}$ is the form defined on $H^1(\R^d;\C)\times H^1(\R^d;\C)$ by
\begin{equation*}
\mathfrak{b}(u,v)=\int_{\R^d}\langle  Q(x)\nabla u(x),\nabla v(x)\rangle\,dx,\qquad u,v\in	H^1(\R^d;\C).
\end{equation*}
\end{lemma}
\begin{proof}
Properties (a) and (b) follow as in \cite[Lemma 5.10]{ALLR24}. 
	
We only need to show that estimate \eqref{eq:form:inequality} holds true. For this purpose, let ${\uu}\in D(\mathfrak{a})$ and $ f\in H^1(\R^d;\C)$ be such that $|f|\le\|\uu\|$. Since $V$ satisfies \eqref{eq:accretivity:V}, we can show that
\begin{align}
&\sum_{j=1}^m\langle Q\nabla u_j,\nabla(|f|{\rm sign}\thinspace(\uu))_j\rangle+\langle V{\uu},|f|{\rm sign}\thinspace(\uu)\rangle - \langle Q\nabla\|\uu\|,\nabla |f|\rangle\notag \\	
=&\bigg\langle  \frac1{\|\uu\|}\sum_{j=1}^m(u_jQ\nabla u_j)\chi_{\{\uu\ne 0\}}, \nabla |f|\bigg \rangle
+\frac{|f|}{\|\uu\|}\chi_{\{\uu\ne 0\}}\langle V{\uu},{\uu}\rangle - \langle Q\nabla\|\uu\|,\nabla |f|\rangle\notag\\ 
&+ \frac{|f|}{\|\uu\|}\chi_{\{\uu\ne \bm{0}\}} \bigg( \sum_{j=1}^m\langle Q\nabla u_j, \nabla u_j\rangle -\bigg\langle \frac{1}{\|\uu\|}\sum_{j=1}^m(u_j Q\nabla u_j)\chi_{\{\uu\ne \bm{0}\}}, \nabla\|\uu\|\bigg\rangle\bigg)\notag\\
=& \frac{|f|}{\|\uu\|}\chi_{\{\uu\ne 0\}}\langle V{\uu},{\uu}\rangle + \frac{|f|}{\|\uu\|}\chi_{\{\uu\ne \bm{0}\}} \bigg( \sum_{j=1}^m\langle Q\nabla u_j, \nabla u_j\rangle -\langle Q\nabla\|\uu\|,\nabla\|\uu\|\rangle\bigg)\notag\\
\ge&\  \frac{|f|}{\|\uu\|}\chi_{\{\uu\ne 0\}}\bigg (\sum_{j=1}^m\langle Q\nabla u_j, \nabla u_j\rangle -\langle Q\nabla \|\uu\|, \nabla\|\uu\|\rangle\bigg).
\label{roxette}
\end{align}

Note that $\nabla\|\uu\|=\frac{1}{\|\uu\|}\sum_{j=1}^m{\rm Re}(u_j\nabla u_j)\chi_{\{\uu\neq\bm{0}\}}$. Hence, setting $u_j=v_j+iw_j$ for every $j\in\{1,\ldots,m\}$, it follows that
\begin{align*}
\langle Q\nabla\|\uu\|,\nabla\|\uu\|\rangle
=&\frac{1}{\|\uu\|^2}\sum_{j,k=1}^m\langle Q\nabla v_j,\nabla v_k\rangle v_jv_k
+\frac{1}{\|\uu\|^2}\sum_{j,k=1}^m\langle Q\nabla w_j,\nabla w_k\rangle w_jw_k\\
&-\frac{2}{\|\uu\|^2}\sum_{j,k=1}^m\langle Q\nabla v_j,\nabla w_k\rangle v_jw_k\\
\le &\frac{1}{\|\uu\|^2}\bigg (\sum_{j=1}^m\|Q^{\frac{1}{2}}\nabla v_j\||v_j|\bigg )^2
+\frac{1}{\|\uu\|^2}\bigg (\sum_{j=1}^m\|Q^{\frac{1}{2}}\nabla w_j\||w_j|\bigg )^2\\
&+\frac{2}{\|\uu\|^2}\bigg (\sum_{j=1}^m\|Q^{\frac{1}{2}}\nabla v_j\||v_j|\bigg )\bigg (\sum_{j=1}^m\|Q^{\frac{1}{2}}\nabla w_j\||w_j|\bigg )\\
=&\frac{1}{\|\uu\|^2}\bigg (\sum_{j=1}^m(\|Q^{\frac{1}{2}}\nabla v_j\||v_j|+\|Q^{\frac{1}{2}}\nabla w_j\||w_j|)\bigg )^2\\
\le & \sum_{j=1}^m(\|Q^{\frac{1}{2}}\nabla v_j\|^2+\|Q^{\frac{1}{2}}\nabla w_j\|^2)\\
= & \sum_{j=1}^m\langle Q\nabla u_j,\nabla u_j\rangle,
\end{align*}
so that the last side of \eqref{roxette} is nonnegative.
Starting from \eqref{roxette} and integrating over $\R^d$ its first side, we deduce that 
\begin{equation*}
\mathfrak{a}({\uu},|f|\,{\rm sign}(\uu)) - \mathfrak{b}(\|\uu\|,|f|)\ge0  
\end{equation*}
and the proof is complete.
\end{proof}

\begin{prop}
Assume that Hypotheses $\ref{hyp_0}$ hold true. Then, for every $t>0$, there exists a family of kernels
$(p_{hk}(t,\cdot ,\cdot))_{1\le h,k\le m}\in L^{\infty}(\R^d\times \R^d,\R^{m\times m})$, such that
\begin{equation*}
    (e^{-tL}\bm{u}(x))_h=\sum_{k=1}^m\int_{\R^d} p_{hk}(t,x,y)u_k(y)dy,\qquad\;\, h\in\{1,\ldots,m\},
\end{equation*}
for every $t>0$, every $\uu\in L^2(\R^d;\C^m)$ and almost every $x\in \R^{d}$. Moreover, 
\begin{equation*}
|p_{hk}(t,x,y)|\le Ct^{-\frac{d}{2}}\exp\left(-\frac{|x-y|^2}{4\gamma t}\right), \quad h,k\in\{1,\ldots,m\},
\end{equation*}
for every $t>0$, almost every $(x,y)\in \R^{d}\times\R^d$, and some positive constants $C$ and $\gamma$ depending only on $\nu,\,d$ and $\max_{i,j\in\{1,\ldots,d\}}\|q_{ij}\|_\infty$.
\end{prop}

\begin{proof}
Combining \cite[Theorem 2.30]{Ou05} and Lemma \ref{lem:dominazione}, we deduce that
\begin{equation}\label{domin-kernel}
\|(e^{-tL}\bm{u})(x)\|\le e^{-tL_0}\|\bm{u}(x)\|
\end{equation}
for every $\bm{u}\in L^2(\R^d;\C^m)$, every $t>0$ and almost every $x\in\R^d$, where $L_0$ is the sectorial operator associated in $L^2(\R^d;\C)$ to the form $\mathfrak{b}$ in Lemma \ref{lem:dominazione}. Moreover, it is known that, for every $t>0$ and every $f\in L^1(\R^d;\C)$, there exists a kernel $q(t,\cdot,\cdot)\in L^\infty(\R^d\times\R^d;\R)$ such that
\begin{equation*}
(e^{-tL_0}f)(x)=\int_{\R^d}q(t,x,y)f(y)dy    
\end{equation*}
for almost every $x\in\R^d$ and 
\begin{equation}\label{Gaussian-q}
0\le q(t,x,y)\le Ct^{-\frac{d}{2}}\exp\left(-\frac{|x-y|^2}{4\gamma t}\right)
\end{equation}
for almost every $(x,y)\in \R^d\times \R^d$ and some positive constants $C$ and $\gamma$ depending only on $\nu,\,d$ and $\max_{i,j\in\{1,\ldots,d\}}\|q_{ij}\|_\infty$ (see  \cite[Theorem 6.10]{Ou05}). In particular, the semigroup $(e^{-tL_0})_{t\ge 0}$ is ultracontractive and this implies that  
\begin{equation}\label{eq:e^-tH:L1:Linf}
\|(e^{-tL}\f)(x)\|\le (e^{-tL_0}\|\f\|)(x) \le \|e^{-tL_0}\|\f\|\|_\infty\le Mt^{-\frac d2}\|\f\|_{1}
\end{equation}
for every $\f\in C_c^\infty(\R^d;\C^m)$, some positive constant $M$ and almost every $x\in\R^d$. Thus, for every $t>0$, $e^{-tL}$ maps $L^1(\R^d;\C^m)$ into $L^\infty(\R^d;\C^m)$ and the action of this semigroup on $\bm{u}\in L^2(\R^d;\C^m)$ can be written as
\begin{equation*}
e^{-tL}\bm{u}=\sum_{h,k=1}^mT_{hk}(t)u_k{\bm e}_h,\qquad t>0,
\end{equation*}
where $T_{hk}=\langle e^{-tL}(\cdot\,{\bm e}_k),{\bm e}_h\rangle$ for every $h,k\in\{1,\dots,m\}$. Observe that, for almost every $x\in\R^d$ and every smooth function $g\in C_c^\infty(\R^d;\C)$, estimate \eqref{eq:e^-tH:L1:Linf} yields 
\begin{equation*}
|(T_{hk}(t)g)(x)| \le \|e^{-tL}(g{\bm e}_k)\|_{\infty} \le Mt^{-\frac d2}\|g\|_1.
\end{equation*}
So, by the Dunford-Pettis theorem, for every $h,k\in\{1,\ldots,m\}$ and every $t>0$ there exists a function $p_{hk}(t,\cdot ,\cdot)\in L^\infty(\R^d\times \R^d;\R)$ such that
\begin{equation*}
(T_{hk}(t)g)(x)=\int_{\R^d}p_{hk}(t,x,y)g(y)dy,\quad \textit{for a.e. } x\in \R^d,\;\,g\in L^1(\R^d;\C).  
\end{equation*}
Now, by \eqref{domin-kernel}, \eqref{Gaussian-q} and the proof of \cite[Theorems 3.3.4 \& 3.3.5]{Meyer-Nieberg}, we infer that
\begin{equation*}
|p_{hk}(t,x,y)|\le Ct^{-\frac{d}{2}}\exp\left(-\frac{|x-y|^2}{4\gamma t}\right), \qquad\;\, h,k\in\{1,\ldots,m\},
\end{equation*}
for every $t>0$ and almost every $(x,y)\in \R^d\times \R^d$.
\end{proof}

We have proved so far that $L$ satisfies Hypotheses \ref{hyp:operatore} of the previous section and we recover the following results. Here, $(p^k_{ij}(t,\cdot,\cdot))_{1\le i,j\le m}$ denotes the family of kernels associated to the operator $(-L)^ke^{-tL}$, for every $t>0$.

\begin{prop}
\label{prop:cplx_kernel_est}
For every $\beta\in\left(0,\frac\pi2\right)$, the following properties are satisfied.
\begin{enumerate}[\rm(i)]
\item There exists a positive constant $c_{\beta}$ such that, for every $i,j\in\{1,\ldots,m\}$ and every $z\in \Sigma_{\beta}$, 
\begin{align*}
|p_{ij}(z,x,y)|\leq c_{\beta}({\rm Re}(z))^{-\frac d2}\exp\left(-\frac{|x-y|^2}{32\gamma(1+\tan(\beta)){\rm Re}(z)}\sin\left(\frac\pi4-\frac\beta2\right)\right)
\end{align*}
for almost every $(x,y)\in\R^{d}\times\R^{d}$.
\item There exists a positive constant $c_{1,k,\beta}$ such that, for every $i,j\in\{1,\ldots,m\}$ and every $z\in \Sigma_\beta$,
\begin{align*}
|p^k_{ij}(z,x,y)|\leq c_{1,k,\beta}({\rm Re}(z))^{-\frac d2-k}\exp\left(-\frac{|x-y|^2}{64\gamma(1+\tan(\beta)){\rm Re}(z)}\sin\left(\frac\pi8-\frac\beta4\right)\right)
\end{align*}
for almost every $(x,y)\in\R^{d}\times\R^{d}$.
\end{enumerate}

\end{prop}

\begin{theo}
The operator $L$ admits a bounded $\textup{H}^\infty(\Sigma_\mu)$-calculus on $L^p(\R^d;\C^m)$ for every $p\in(1,\infty)$ and every $\mu\in\left(0,\pi\right)$.
\end{theo}

Since the functions $\xi(z)=z^{is}$ belong to $\textup{H}^\infty(\Sigma_\mu)$ for every $s\in\R$ and every $\mu\in(0,\pi)$, we get the following result.

\begin{cor}
The operator $L$ has bounded imaginary powers in $L^p(\R^d;\C^m)$, i.e., $L^{is}\in \mathscr{L}(L^p(\R^d;\C^m))$ for every $s\in\R$ and there exist positive constants $c$ and $\delta$, which depend on $p$ and $\mu$, such that $\|L^{is}\|_{\mathscr L(L^p(\R^d;\C^m))}\leq ce^{\delta|s|}$ for every $s\in \R$.
\end{cor}

\section*{Acknowledgments}
This article is based upon work from the project ``Elliptic and parabolic problems, heat kernel estimates and spectral theory" CUP D53D23005580006, funded by European Union
Next Generation EU within the PRIN 2022 program (D.D. 104 - 02/02/2022 Ministero dell’Università e della Ricerca). 
The authors are also members of G.N.A.M.P.A. of the Italian Istituto Nazionale di Alta Matematica (I.N.d.A.M.).

\end{document}